\documentclass[10pt,a4paper]{article}
\usepackage{amsmath, amsthm}
\usepackage{esint}
\usepackage{amsfonts}
\usepackage{amssymb}
\usepackage{hyperref}
\usepackage{graphicx}
\usepackage[T1]{fontenc}
\usepackage[francais,english]{babel}
\usepackage{listings}
\usepackage{vmargin} 
\usepackage{dsfont}
\usepackage{color}
\numberwithin{equation}{section}%numéroter les référence par section  
\setmarginsrb{2cm}{2.5cm}{2cm}{2.5cm}{0cm}{0cm}{0cm}{1cm} % redéfinir les marges

\newtheorem{de}{Definition}[section]
\newtheorem{thm}[de]{Theorem}
\newtheorem{cor}[de]{Corollary}
\newtheorem{prop}[de]{Proposition}
\newtheorem{lem}[de]{Lemma}
\newtheorem{rem}[de]{Remark}

\newcommand\E{\mathcal{E}}
\newcommand\F{\mathcal{F}}
\newcommand\V{\mathcal{V}}

\newcommand\G{\mathcal{G}}
\newcommand\C{\mathcal{C}}

\newcommand\N{\mathbb{N}}
\newcommand\M{\mathcal{M}}

\newcommand\trho{\tilde{\rho}}

\renewcommand{\textbf}[1]{\begingroup\bfseries\mathversion{bold}#1\endgroup}

\DeclareMathOperator*{\dive}{div}
\DeclareMathOperator*{\argmin}{argmin}
\DeclareMathOperator*{\Rn}{\mathbb{R}^\textit{n}}
\DeclareMathOperator*{\R}{\mathbb{R}}
\DeclareMathOperator*{\Pa}{\mathcal{P}}

\DeclareMathOperator*{\Paad}{\mathcal{P}^{ac}_2}

\newcommand\loc{\mathop{\mathrm{loc}}\nolimits}
\newcommand\id{\mathop{\mathrm{id}}\nolimits}

\title{A splitting method for nonlinear diffusions with nonlocal, nonpotential drifts}

\author {Guillaume Carlier \thanks{\scriptsize CEREMADE, UMR CNRS 7534, Universit\'e Paris IX
Dauphine, Pl. de Lattre de Tassigny, 75775 Paris Cedex 16, FRANCE
\texttt{carlier@ceremade.dauphine.fr}},
Maxime Laborde \thanks{\scriptsize CEREMADE, UMR CNRS 7534, Universit\'e Paris IX
Dauphine, Pl. de Lattre de Tassigny, 75775 Paris Cedex 16, FRANCE
\texttt{laborde@ceremade.dauphine.fr}.}}

\begin{document}

\maketitle

\begin{abstract}
We prove an existence result for nonlinear diffusion equations in the presence of a nonlocal density-dependent drift which is not necessarily potential. The proof is constructive and based on the Helmholtz decomposition of the drift and a splitting scheme. The splitting scheme combines transport steps by the divergence-free part of the drift and semi-implicit minimization steps \`a la Jordan-Kinderlherer Otto to deal with the potential part.

\end{abstract}


\medskip

\textbf{MS Classification:} 35K15, 35K40.

\section{Introduction}\label{sec-intro}

In their seminal paper \cite{JKO}, Jordan-Kinderlehrer and Otto, showed that the Fokker-Planck equation
\[\partial_t \rho -\Delta \rho- \dive(\rho \nabla V)=0, \qquad \rho_{|t=0} = \rho_0\]
with a probability density $\rho_0$ as Cauchy datum can be viewed as the gradient flow for the Wasserstein metric of the relative entropy with respect to the equilibrium measure $\rho_\infty:=e^{-V}$, $E(\rho\vert \rho_\infty):= \int \rho \log(\rho/ \rho_\infty)$. They also introduced an implicit Euler-scheme (nowadays referred to as the JKO scheme): given a time step $h>0$ and starting from $\rho_h^0=\rho_0$, construct inductively a sequence $\rho_h^k$ by:
\[\rho_h^{k+1}\in \argmin \left\{\frac{1}{2h} W_2^2(\rho, \rho_h^k)+E(\rho\vert \rho_\infty)  \right\}\]
where $W_2$ denotes the $2$-Wasserstein distance (see section \ref{sec-prel}) and proved its convergence to the solution of the Fokker-Planck equation as $h\to 0$.  The theory of Wasserstein gradient flows has developed rapidly in the last twenty years with many applications for instance to porous medium equations \cite{O} or  aggregation equations \cite{CDFLS}. The reference textbook of Ambrosio, Gigli and Savar\'e \cite{AGS} gives a very detailed account of this powerful theory, which enables to study general nonlinear diffusion equations of the form
\[\partial_t \rho -\Delta P(\rho)- \dive(\rho \nabla V)=0 \mbox{ where } P(\rho):=\rho F'(\rho)-F(\rho),  \]
as Wasserstein gradient flows for the energy $\int (F(\rho)+V \rho)$. 

\medskip

The purpose of the present paper is to present a splitting transport-JKO scheme to study nonlinear diffusion equations (or more generally, systems) with a general density-dependent drift:
\begin{eqnarray*}
&\partial_t \rho(t,x)-\Delta P(\rho(t,x))    - \dive(\rho(t,x) U[\rho(t,.)](x))=0, \; t\ge 0, \; x\in \Omega, \\
&(\nabla P(\rho)+U[\rho])\cdot \nu=0 \mbox{ on $\partial \Omega$},\; \rho_{|t=0} = \rho_0,
\end{eqnarray*}
where for every probability density $\rho$, $U[\rho]$ is a-not necessarily potential- vector field, for instance, in even dimensions, it can mix a gradient and Hamiltonian structures i.e. be of the form $\nabla V[\rho]+ J  \nabla H[\rho]$ (where $J$ is the usual symplectic matrix). The potential case where $U[\rho]=\nabla V[\rho]$ can be studied by means of a semi-implicit JKO scheme introduced by Di Francesco and Fagioli \cite{DFF} in the nondiffusive case and further developed by Laborde \cite{L} for the case of a non linear diffusion.  The idea of our splitting scheme is natural and consists in performing a Helmholtz decomposition of $U[\rho]$. We then treat the divergence-free part purely by (continuous in time) transport and the potential part by the semi-implicit JKO scheme. For the transport steps of the splitting scheme, we essentially need the divergence-free part to have some Sobolev regularity in $x$ so as to be able to apply DiPerna-Lions theory, we will also need both the potential and divergence-free part of $\rho \mapsto U[\rho]$ to satisfy some Lipschitz continuity condition with respect to the Wasserstein distance. In our recent work \cite{CL}, we studied the same type of equations or systems (in the periodic in space setting) by quite different arguments (approximation  by uniformly parabolic equations). One advantage of the constructive splitting method presented here is that  the transport steps by a divergence-free vector field preserve the internal energy, this is one way to overcome some difficulties discussed in \cite{MesSan} (section 5, variant 3).

\medskip

The paper is organized as follows. Section \ref{sec-prel} recalls some results from optimal transport and DiPerna Lions theory. Section \ref{sec-assump} lists the various assumptions, explains the splitting scheme and gives the main result. Section \ref{sec-estimates} gives estimates on the discrete sequences of measures obtained by the splitting scheme. Convergence of the scheme as the time step goes to $0$ to a solution of the PDE is proved in section \ref{sec-proof}. In the concluding section \ref{sec-concl}, we briefly  discuss extension to systems and uniqueness issues.

\section{Preliminaries}\label{sec-prel}

\subsection{Wasserstein space}

We recall some results from optimal transport theory that we will use in the sequel, we refer the reader to the textbooks of Villani \cite{V1, V2}, Ambrosio, Gigli and Savar\'e \cite{AGS}  or Santambrogio \cite{S} for a detailed exposition. 
Let $\Omega$ be an open subset of $\Rn$, we denote by $\Pa(\Omega)$ the set of Borel probability measures on $\Omega$, $\Pa_2(\Omega):=\{\rho \in \Pa(\Omega) \; : \; M(\rho):=\int_{\Omega} \vert x \vert^2 d \rho(x) <+\infty\}$ and $\Paad(\Omega)$ the elements of $\Pa_2(\Omega)$ which in addition are absolutely continous with respect to the $n$-dimensional Lebesgue measure. Given $\rho_0$ and $\rho_1$ in $\Pa_2(\R^n)$ and denoting by $\Pi(\rho_0, \rho_1)$ the set of transport plans between $\rho_0$ and $\rho_1$ i.e. the set of Borel probability measures on $\Rn\times \Rn$ having $\rho_0$ and $\rho_1$ as marginals, the $2$-Wasserstein distance between $\rho_0$ and $\rho_1$, $W_2(\rho_0, \rho_1)$ is defined as the value of the optimal transport problem
\[W_2^2(\rho_0, \rho_1)=\inf_{\gamma\in \Pi(\rho_0, \rho_1)} \int_{\Rn\times \Rn} \vert x-y \vert^2 d \gamma(x,y).\]
Since this is a linear programming problem, it admits a dual formulation, which reads
\[\frac{1}{2}W_2^2(\rho_0, \rho_1)=\supÊ \Big\{ \int_{\Rn} \varphi d \rho_0 + \int_{\R^n} \psi d \rho_1 \; : \; \varphi(x)+\psi(y) \le \frac{1}{2} \vert x-y \vert^2\Big\}.\]
Optimal potentials in the problem above are called Kantorovich potentials between $\rho_0$ and $\rho_1$, they can be taken semi-concave and their existence is well-known (see \cite{V1, V2, AGS, S}). If $\rho_0 \in    \Paad(\Omega)$ a celebrated result of Brenier \cite{B} states that there is a unique optimal plan $\gamma$ between $\rho_0$ and $\rho_1$ and it is induced by a transport map $T$ i.e. is of the form $\gamma=(\id, T)_\# \rho_0$ and optimality is characterized by the fact that $T=\nabla u$ with $u$ convex, moreover $T(x)=x-\nabla \varphi(x)$ where $\varphi$ is the Kantorovich potential between $\rho_0$ and $\rho_1$. This in particular gives
\[W_2^2(\rho_0, \rho_1)=\int_{\Rn} \vert \nabla \varphi(x)\vert^2 \rho_0(x) dx.\]
Another important result is the Benamou-Brenier formula \cite{BB} which gives a dynamic formulation of $W_2^2(\rho_0, \rho_1)$ and expresses it as the infimum of the kinetic energy:
\[\int_0^1 \int_{\Rn} \vert v_t(x)\vert^2 d \rho_t(x) dt\]
among solutions of the continuity equation
\[\partial_t \rho + \dive(\rho v)=0, \;  \qquad \rho_{|t=0} = \rho_0, \;  \qquad \rho_{|t=1} = \rho_1.\]

%(resp. $\Pa_2(\Omega)$, $\Paad(\Omega)$)  (resp. the subset of $\Pa_2(\Omega)$ consisting of measures having a finite second moment 

\subsection{Flows of weakly differentiable vector fields}

We will also need to  apply the DiPerna Lions theory \cite{DPL} in the special case of divergence-free vector fields. Let $W\in W^{1,1}_{\loc}(\Rn, \Rn)$ be divergence-free $\dive(W)=0$ and with at most linear growth
\[\vert W(x)\vert \leqslant C(1+\vert x \vert).\]
Then there exists a unique flow map $X$ : $\R_+ \times \Rn\to \Rn$, $X\in C(\R_+, L^1_{\loc}(\Rn))$ such that
\begin{itemize}

\item $X(0,.)=\id$ and for a.e. $x$, $t\in \R_+\mapsto X(.,x)$ is a solution of the ODE $\dot{X}=W(X)$ i.e.:
\[X(t,x)=x+\int_0^t W(X(s,x)) ds, \; t\ge 0,\]

\item $X$ satisfies the group property $X(t, X(s,x))=X(t+s,x)$ for a.e. $x$ and every $t, s \ge 0$,

\item for every $X(t,.)$ preserves the $n$-dimensional Lebesgue measure.

\end{itemize}

Moreover given $\rho_0\in \Paad(\Rn)$, $\rho(t,.):=X(t,.)_\# \rho_0=\rho_0(X(t,.)^{-1})$ is the unique weak solution of the continuity equation
\begin{equation}\label{cauchytW}
\partial_t \rho+\dive(\rho W)=0, \; \qquad \rho_{|t=0} = \rho_0,
\end{equation}
which, since $W$ is divergence-free, can also be rewritten as the transport equation $\partial_t \rho +\nabla \rho \cdot W=0$. If we are given an open subset $\Omega$ of $\Rn$ with a smooth boundary and  which is tangential (in the sense of traces) to $\partial \Omega$, then the DiPerna-Lions flow $X$ leaves $\Omega$ invariant so that if $\rho_0\in \Paad(\Omega)$ (extended by $0$ outside $\Omega$, say), the solution $\rho_t=X(t,.)_\#\rho_0$ of \eqref{cauchytW} remains supported in $\Omega$ hence may be viewed as a curve with values in $\Paad(\Omega)$.

\section{Assumptions and main result}\label{sec-assump}

Given a suitable convex nonlinearity $F$ and its associated pressure $P(\rho)=\rho F'(\rho)-F(\rho)$ as well as a nonlocal drift $\rho \mapsto U[\rho]$, our goal is to solve 
\begin{eqnarray}
\label{eqution P2}
\partial_t \rho - \Delta P(\rho)- \dive(\rho U [\rho] ) = 0, \qquad \rho_{|t=0} = \rho_0,
\end{eqnarray}
on $(0,+\infty) \times \Omega$, where $\Omega$ is a smooth domain  of $\Rn$ (not necessary bounded), in case $\Omega$ has a boundary, the previous equation is supplemented with the no-flux boundary condition ($\nu$ denotes the outer unit normal to $\partial \Omega$):
\begin{equation}\label{no-flux}
(\nabla P(\rho)+ U[\rho]\rho)\cdot \nu =0 \mbox{ on } \partial \Omega.
\end{equation}

\smallskip

For every $\rho \in \Pa(\Omega)$, we assume that  the Helmholtz  decomposition of the vector field $U[\rho]$:
\begin{equation}
\label{helmholtz decomposition}
U[\rho]= -W[\rho] + \nabla V[\rho],
\end{equation}
with 
\[\nabla \cdot W[\rho] =0, \; W[\rho]\cdot \nu=0 \mbox{ on } \partial \Omega,\]
satisfies the following assumptions.

\bigskip

{\bf{Assumptions on the potential part $V$}}:
\begin{itemize}

\item $\nabla V[\rho] \in L^\infty_{\loc}$ uniformly in $\rho$ i.e for all $K \subset \subset \Omega$, there exists $C>0$ such that for all $\rho \in \Pa(\Omega)$,
\begin{equation}
\label{hyp V Linffty bound}
\| \nabla V[\rho] \|_{\infty,K} \leqslant C.
\end{equation} 
Note that by Rademacher's condition, this condition implies that $V[\rho]$ is differentiable a.e.,

\item $V[\rho]$ is semi-convex uniformly in $\rho$ i.e there exists $C$ such that for all $\rho \in \Pa(\Omega)$, for every $y \in \Omega$ and every $x\in \Omega$, point of differentiability of $V[\rho]$:
\begin{equation}
\label{hyp V semi-convex}
V[\rho](y) \geqslant V[\rho](x) + \langle \nabla V[\rho](x) , y-x \rangle -\frac{C}{2}|y-x|^2,
\end{equation}

%{\color{red}  j ai une peu change l ordre des conditions pour donner un sens au gradient dans la semi convexite}

\item there exists $C \geqslant 0$ such that for all $\rho \in \Pa(\Omega)$, for all $x \in \Omega$:
\begin{equation}
\label{hyp V bounded from below}
V[\rho](x) \geqslant - C(1+|x| ), 
\end{equation}

\item $\nabla V[\rho] \in L^2(\rho)$ uniformly in $\rho$, i.e there exists $C>0$ such that for all $\rho \in \Pa(\Omega)$,
\begin{equation}
\label{hyp V L2 bound}
\int_\Omega |\nabla V[\rho] |^2 \, d\rho \leqslant C,
\end{equation}
\item There exists $C>0$ such that for all $\rho,\mu \in \Pa(\Omega)$,
\begin{equation}
\label{hyp V cv L2}
\int_\Omega |\nabla V[\rho] - \nabla V[\mu] |^2 \, d\rho \leqslant C W_2^2(\rho,\mu),
\end{equation}

\end{itemize}

{\bf{Assumptions on the divergence-free part $W$}}: 
\begin{itemize}
\item there exists $C>0$ such that for all $\rho \in \Pa(\Omega)$ and 
\begin{equation}
\label{hyp W regularity + Linfty bound}
W[\rho] \in W^{1,1}_{\loc}(\Rn) \qquad \text{ and }  \qquad \vert  W[\rho](x) \vert \leqslant C(1+|x|), \mbox{ for all $x \in \Rn$}
\end{equation} 

%\red{Ca n est pas clair que Di Perna Lions marche sur un domaine avec des derivees qui peuvent mal se comporter pres du bord, $W^{1,1}_{\loc}(\Rn)$ ou $W^{1,1}(\Omega)$ avec $\partial \Omega$ borne marchent, j ai prefere garde la premiere solution}

\item There exists $C>0$ such that for all $\rho,\mu \in \Pa(\Omega)$,
\begin{equation}
\label{hyp W cv L2}
\int_\Omega | W[\rho] - W[\mu] |^2 \, d\rho \leqslant C W_2^2(\rho,\mu),
\end{equation}
\end{itemize}

{\bf{Assumptions on the internal energy $F$ and the associated pressure $P$}}:

The nonlinear diffusion term is given by a continuous strictly convex superlinear (i.e. $F(\rho)/\rho\to +\infty$ as $\rho\to +\infty$) function $F \, : \, \R^+ \rightarrow \R $ of class $\C^2((0,+\infty))$ which satisfies
\begin{equation}
\label{hyp F}
F(0)=0, \text{  and  } P(\rho) \leqslant C(\rho+F(\rho)).
\end{equation}
where $P(\rho):=\rho F'(\rho)-F(\rho)$ is the pressure associated to $F$. Moreover, we define $\F \, : \, \Pa(\Omega)\rightarrow \R$ by
$$ \F(\rho):=\left\{\begin{array}{ll}
\int_{\Omega} F(\rho(x)) \, dx  & \text{  if } \rho \ll \mathcal{L}^n,\\
+\infty & \text{  otherwise. }
\end{array}\right.$$
And we assume that

\begin{equation}
\label{hyp F lower bound }
\F(\rho) \geqslant -C(1+M(\rho))^\alpha, \qquad \text{ for all } \rho \in \Pa(\Omega),
\end{equation}
where $\alpha \in (0,1)$ and $M(\rho):= \int_{\Omega} |x|^2 \,d\rho(x)$ is the second moment of $\rho$.

\smallskip

The typical examples of energies we have in mind are $F(\rho):=\rho\log(\rho)$, which gives a linear diffusion driven by the laplacian, and $F(\rho):=\rho^m$ ($m>1$) which corresponds to the porous medium equation.

\bigskip

A weak solution of \eqref{eqution P2}-\eqref{no-flux} is a curve $\rho \, : \, t\in (0,+\infty) \mapsto \rho(t,\cdot) \in \Paad(\Omega)$ such that $\nabla P(\rho) \in \M^n ([0,+\infty) \times \Omega)$ and
\begin{equation}
\label{def sol}
\int_0^\infty \left( \int_\Omega (\partial_t \phi \rho -\nabla \phi \cdot U[\rho]\rho)dx-\int_\Omega \nabla \phi \cdot d\nabla P(\rho) \right) dt =-\int_\Omega \phi(0,x)\rho_0(x) \, dx,
\end{equation}
for every $\phi \in \C^\infty_c([0,+\infty)\times \Rn)$.

Our main result is the following:

\begin{thm}
\label{theorem existence P2}
Assume $\rho_0 \in \Paad(\Omega)$ such that 
\begin{equation}
\label{hyp CI}
\F(\rho_0) < +\infty,
\end{equation}
then \eqref{eqution P2} admits at least one weak solution.
\end{thm}

The proof of this theorem is given in the next sections and is based on the following splitting scheme that combines pure transport steps by the divergence-free part of the drift $U$ and  Wasserstein gradient flow steps taking into account the potential $V$ in a semi-implicit way. More precisely, given a time step $h>0$, we construct by induction a sequence $\rho_h^k \in \Paad(\Omega)$ by setting $\rho_h^0=\rho_0$ and given $\rho_h^k$ we find $\rho_h^{k+1}$ using the following scheme:

\begin{itemize}
\item \textbf{pure transport phase:} we introduce an intermediate measure, $\tilde{\rho}_h^{k+1}$ (with $\tilde{\rho}_h^0=\rho_0$) defined by
\begin{equation}
\label{scheme free transport}
\tilde{\rho}_h^{k+1} = {X_h^k(h,\cdot)}_{\#}\rho_h^k,
\end{equation}
where $X_h^k$ is solution of 
\begin{eqnarray}
\label{eqution transport}
\left\{\begin{array}{l}
\partial_t X_h^k = W[\rho_h^k] \circ X_h^k,\\
X_h^k(0,\cdot) = \id.
\end{array}\right.
\end{eqnarray}
Since $W[\rho_h^k]$ satisfies \eqref{hyp W regularity + Linfty bound}, as recalled in section \ref{sec-prel},  DiPerna-Lions theory \cite{DPL} implies that $X_h^k$ is well defined. Moreover, since $W[\rho_h^k]$ is divergence-free then $X_h^k$ preserves the Lebesgue measure and leaves the domain $\Omega$ invariant thanks to the fact that $W[\rho]$ is tangential to $\partial \Omega$. Therefore $\tilde{\rho}_h^{k+1}=\rho_h^k( {X_h^k}^{-1} )$ which implies the conservation of the internal energy:
\begin{eqnarray}
\F(\tilde{\rho}_h^{k+1}) &= & \int_\Omega F(\rho_h^k( {X_h^k}^{-1}(x) )) \, dx \nonumber \\
&=&\int_\Omega F(\rho_h^k( x )) \, dx = \F(\rho_h^k)  \label{conservation energy}
\end{eqnarray}

In addition, we can see $\tilde{\rho}_h^{k+1}$ is the value at time $h$ of the solution $\mu$ of the continuity equation
\begin{eqnarray}
\label{scheme continuity equation}
\left\{ \begin{array}{l}
\partial_t \mu +\dive(\mu W[\rho_h^k])=0,\\
\mu_{|t=0} = \rho_h^k.
\end{array}\right.
\end{eqnarray}

Thanks to these observations, we can easily  control the $W_2$-distance between $\tilde{\rho}_h^{k+1}$ and $\rho_h^k$. Indeed, using Benamou-Brenier formula and \eqref{hyp W regularity + Linfty bound}, we obtain
\begin{eqnarray*}
W_2^2(\tilde{\rho}_h^{k+1},\rho_h^k) & \leqslant & h\int_0^h \int_\Omega |W[\rho_h^k]|^2 \,d \mu_t dt\\
&\leqslant & Ch\int_0^h \int_\Omega (1+|x|^2) \,d \mu_t dt\\
&\leqslant & Ch\int_0^h \left( 1+M(\mu_t) \right) \,dt
\end{eqnarray*}
Moreover,
\begin{eqnarray*}
\frac{d}{dt}M(\mu_t) & = & \int_\Omega |x|^2 \partial_t \mu_t\\
& = & -\int_\Omega |x|^2 \dive(W[\rho_h^k] \mu_t)\\
& = & 2\int_\Omega x \cdot W[\rho_h^k] \mu_t\\
&\leqslant & C(M(\mu_t) +1).
\end{eqnarray*}
We obtain the last line using \eqref{hyp W regularity + Linfty bound}, Cauchy-Schwarz inequality and Young's inequality. Then,
 $$ M(\mu_t) \leqslant C(t+1)e^{Ct}\leqslant 2Ce^C \text{  for } t\leqslant 1,$$
which implies
\begin{equation}
\label{estimation distance transport libre}
W_2^2(\tilde{\rho}_h^{k+1},\rho_h^k) \leqslant Ch^2.
\end{equation} 

%\red{ A justifier un peu mieux peut être}

\item \textbf{semi-implicit JKO scheme:} In the second step we use a semi-implict version of the Jordan-Kinderlehrer-Otto scheme \cite{JKO}, introduced by Di Francesco and Fagioli in \cite{DFF} and used in \cite{L}, with $\tilde{\rho}_h^{k+1}$ being the measure defined in the previous step. More precisely, we select $\rho_h^{k+1}$ as a solution of
\begin{equation}
\label{scheme semi-impl JKO}
\inf_{\rho \in \Paad(\Omega) } \E_h(\rho |\trho_h^{k+1}) :=  W_2^2(\rho,\tilde{\rho}_h^{k+1}) +2h\left( \F(\rho) +\V(\rho|\tilde{\rho}_h^{k+1}) \right) ,
\end{equation} 
where
$$ \V(\rho|\mu):=\int_\Omega V[\mu] \,d\rho.$$
By standard compactness and lower semicontinuity argument, \eqref{scheme semi-impl JKO} admits at least one solution (see for example \cite{JKO,L}) so the sequence $\rho_h^k$ is well defined (it is even actually unique  by strict convexity of $ \E_h(. |\trho_h^{k+1})$).
\end{itemize}

To summarize, given a time step $h>0$, we construct by induction two sequences $\rho_h^k$ and $\trho_h^k$ with the following splitting scheme: $\rho_h^0=\trho_h^0=\rho_0$ and for all $k \geqslant 0$,
\begin{eqnarray}
\label{sechem entier}
\left\{ \begin{array}{l}
\trho_h^{k+1} = X_h^k(h,\cdot)_{\#} \rho_h^k,\\
\rho_h^{k+1} \in \argmin_{\rho \in \Paad(\Omega)} \left\{ W_2^2(\rho,\tilde{\rho}_h^{k+1}) +2h\left( \F(\rho) +\V(\rho|\tilde{\rho}_h^{k+1}) \right) \right\}.
\end{array}\right.
\end{eqnarray}

We finally introduce three different interpolations:
\begin{itemize}
\item We denote $\rho_h$ the usual piecewise constant interpolation of the sequence $\rho_h^k$
\begin{equation}
\label{interpolation usual JKO}
\rho_h(t,\cdot):= \rho_h^{k+1} \qquad \text{ if } t \in (hk,h(k+1)],
\end{equation}

\item similarly, we interpolate in a piecewise constant way the sequence $\trho_h^k$:
\begin{equation}
\label{interpolation perturbation 1}
\trho_h^1(t,\cdot):= \trho_h^{k+1} \qquad \text{ if } t \in (hk,h(k+1)],
\end{equation}

\item finally, we denote by $\trho_{h}^2$ the \emph{continuous} interpolation of $\trho_h^k$
\begin{equation}
\label{interpolation perturbation 2}
\trho_h^2(t,\cdot):= X_h^k(t-hk,\cdot)_{\#}\rho_h^k \qquad \text{ if } t \in (hk,h(k+1)].
\end{equation}
We remark that on $(hk,h(k+1)]$, $\trho_h^2$ is the solution on $(0,h)$ of the continuity equation \eqref{scheme continuity equation}.

\end{itemize}

%\red{explication stratégie preuve}

The next two sections are devoted to the proof of theorem \ref{eqution P2}. In section \ref{sec-estimates}, we derive various estimates on the sequences generated by the splitting scheme above, in particular thanks to the Euler-Lagrange equation of the semi-implicit JKO steps. This enables us to pass to the limit as the time step goes to $0$ (the difficult term being of course the nonlinear pressure term) and thus to conclude the existence proof, this is done in section \ref{sec-proof}.

\section{Estimates}\label{sec-estimates}

%In this section, we detail the proof of theorem \ref{theorem existence P2}. In the first part, we prove standard a priori estimates. \red{etc à détailler}.

\subsection{Basic a priori estimates}
Using the semi-implicit JKO scheme we first obtain the following  a priori estimates on $\rho_h$, $\trho_{h}^1$ and $\trho_{h}^2$.

\begin{prop}
There exists $h_0>0$, such that for $T>0$,  there exists $C>0$ such that, for all $h,k$, with $h\in (0, h_0)$ and  $hk<T$, $N=\lceil \frac{T}{h } \rceil$, we have
\begin{eqnarray}
& M(\rho_h^k) \leqslant C, \label{estimation moment}\\
& \F(\rho_{h}^k) \leqslant C,\label{estimation fonctional}\\
&\sum_{k=0}^{N-1} W_2^2(\trho_{h}^{k+1},\rho_{h}^{k+1}) \leqslant Ch.\label{estimation distance}
\end{eqnarray}

\end{prop}

\begin{proof}
Using $\trho_h^{k+1}$ as a competitor of $\rho_h^{k+1}$ in  \eqref{scheme semi-impl JKO}, we obtain
\begin{equation}
\label{inequality opt JKO scheme}
\frac{1}{2h}W_2^2(\rho_h^{k+1},\trho_h^{k+1}) \leqslant \F(\trho_h^{k+1})-\F(\rho_h^{k+1}) +\int_\Omega V[\trho_h^{k+1}] \left( \trho_h^{k+1}-\rho_h^{k+1} \right).
\end{equation}

Let $\gamma$ be the optimal transport plan between $\trho_h^{k+1}$ and $\rho_h^{k+1}$. Then we have
\begin{eqnarray*}
\int_\Omega V[\trho_h^{k+1}] \left( \trho_h^{k+1}-\rho_h^{k+1} \right) & =& \int_\Omega \left( V[\trho_h^{k+1}](x) -V[\trho_h^{k+1}](y) \right) \,d\gamma(x,y)\\
&=& \int_\Omega \left( V[\trho_h^{k+1}](x) -V[\trho_h^{k+1}](y)  + \nabla V[\trho_h^{k+1}](x) \cdot (y-x) \right) \,d\gamma(x,y)\\
&-& \int_\Omega \nabla V[\trho_h^{k+1}](x) \cdot (y-x) \,d\gamma(x,y).
\end{eqnarray*}

Using \eqref{hyp V semi-convex} for the first part and Cauchy-Schwarz inequality and \eqref{hyp V L2 bound} for the second part of the right hand side, we find
\begin{eqnarray*}
\int_\Omega V[\trho_h^{k+1}] \left( \trho_h^{k+1}-\rho_h^{k+1} \right) &\leqslant &  \frac{C}{2}  \int_\Omega |x-y|^2 \,\gamma(x,y)  + C \Big(\int_\Omega |x-y|^2 \,\gamma(x,y)\Big)^{\frac{1}{2}}
    \\
&= & \frac{C}{2}W_2^2(\trho_h^{k+1},\rho_h^{k+1})+ CW_2(\trho_h^{k+1},\rho_h^{k+1}).
\end{eqnarray*}
Choosing $h\le h_0 \le \frac{1}{2C}$ and using Young's inequality
\[W_2(\trho_h^{k+1},\rho_h^{k+1}) \le \frac{1}{8hC} W_2^2(\trho_h^{k+1},\rho_h^{k+1})+ 2 Ch,\]
\eqref{inequality opt JKO scheme} becomes
$$\frac{1}{8h}W_2^2(\rho_h^{k+1},\trho_h^{k+1}) \leqslant \F(\trho_h^{k+1})-\F(\rho_h^{k+1}) +Ch.$$

Now using \eqref{conservation energy}, to recover a telescopic sum, and summing over $k$, we obtain

\begin{eqnarray*}
\sum_{k=0}^{N-1} W_2^2(\rho_h^{k+1},\trho_h^{k+1}) \leqslant 8h \left( \F(\rho_0) - \F(\rho_h^N) +CT \right),
\end{eqnarray*}

this inequality and \eqref{hyp CI} imply \eqref{estimation fonctional}. In addition, since the lower bound of $\F$ is controlled by the second moment, 
\begin{equation}
\label{ineq. pour borne moment}
\sum_{k=0}^{N-1} W_2^2(\rho_h^{k+1},\trho_h^{k+1}) \leqslant 8 h \left( \F(\rho_0) + C(1+M(\rho_h^N))^\alpha +CT \right).
\end{equation}

But, with \eqref{ineq. pour borne moment} and by standard arguments (see \cite{JKO,L}), we deduce that $M(\rho_h^k)$ satisfies \eqref{estimation moment} and then  \eqref{ineq. pour borne moment}, \eqref{hyp CI} and \eqref{estimation moment} give \eqref{estimation distance}. 
\end{proof}

\begin{rem}
\label{estimation trho}
Using  estimate \eqref{estimation distance transport libre} between $\trho_h^{k+1}$ and $\rho_h^k$, and \eqref{estimation distance}, we also have
\begin{eqnarray*}
\sum_{k=0}^{N-1} W_2^2(\rho_h^{k},\rho_h^{k+1}) \leqslant Ch \text{ and } \sum_{k=0}^{N-1} W_2^2(\trho_h^{k},\trho_h^{k+1}) \leqslant Ch.
\end{eqnarray*}
Moreover, using \eqref{conservation energy}, we have for all $t \in [0,T]$,
\begin{eqnarray*}
\F(\trho_h^1(t)),\F(\trho_h^2(t)) \leqslant C \qquad \text{ and } \qquad M(\trho_h^1(t)),M(\trho_h^2(t)) \leqslant C.
\end{eqnarray*}

\end{rem}

\subsection{Discrete Euler-Lagrange equation and stronger estimates}

Let us start with the Euler-Lagrange equation of \eqref{scheme semi-impl JKO}.

\begin{prop}
\label{equality a.e}
For all $k\geqslant 0$, we have $P(\rho_{h}^{k+1}) \in W^{1,1}(\Omega)$ and
\begin{eqnarray}
\label{equality a.e 1}
h\left(\nabla V[\trho_h^{k+1}] \rho_{h}^{k+1}+\nabla P(\rho_{h}^{k+1})\right)=-\nabla \varphi_{h}^{k+1} \rho_{h}^{k+1}  \qquad \text{      a.e},
\end{eqnarray}
where $\varphi_{h}^{k+1}$ is a  Kantorovich potential from $\rho_{h}^{k+1}$ to $\trho_{h}^{k+1}$ (so that its gradient is unique $\rho_{h}^{k+1}$-a.e.) for $W_2$.\end{prop}

\begin{proof}

The proof is the same as in \cite{A,L} for example. We start by taking the first variation in the semi-implicit JKO scheme along the flow of a smooth vector field.
Let $\xi \in \mathcal{C}^\infty_c(\Omega ; \Rn)$ be given and $\Phi_\tau$ the corresponding flow defined by
$$ \partial_\tau \Phi_\tau = \xi \circ \Phi_\tau, \, \Phi_0 =\id.$$
We define a pertubation of $\rho_{h}^{k+1}$ by $\rho_\tau:= {\Phi_\tau}_{\#}\rho_{h}^{k+1}$. 
Then we get
\begin{eqnarray}
\label{first variation}
\frac{1}{\tau}\left( \E_{h}(\rho_\tau |\trho_{h}^{k+1})-\E_{h}(\rho_{h}^{k+1} |\trho_{h}^{k+1}) \right) \geqslant 0.
\end{eqnarray}
By standard computations, we have
\begin{eqnarray}
\label{first variation dist2}
\limsup_{\tau \searrow 0} \frac{1}{\tau}(W_2^2(\rho_\tau,\trho_{h}^{k+1}) -W_2^2(\rho_{h}^{k+1},\trho_{h}^{k+1})) \leqslant \int_{\Omega \times \Omega} (x-y) \cdot \xi(x) \, d\gamma_{h}^{k+1}(x,y),
\end{eqnarray}
with $\gamma_{h}^{k+1}$ is the $W_2$-optimal transport plan in $\Pi (\rho_{h}^{k+1},\trho_{h}^{k+1})$ and $\gamma_{h}^{k+1}=(\id \times T_{h}^{k+1})_{\#}\rho_{h}^{k+1}$ with $T_{h}^{k+1} =\id -\nabla \varphi_{h}^{k+1}$. Moreover, using \eqref{hyp F}, \eqref{estimation fonctional} and Lebesgue's dominated convergence theorem, we obtain
\begin{eqnarray}
\label{first variation entropie}
\limsup_{\tau \searrow 0} \frac{1}{\tau}(\F(\rho_\tau) -\F(\rho_{h}^{k+1}) )\leqslant - \int_{\Omega } P(\rho_{h}^{k+1}(x)) \dive(\xi(x))\, dx.
\end{eqnarray}
Finally,
\begin{eqnarray}
\label{first variation potentiel}
\limsup_{\tau \searrow 0} \frac{1}{\tau}(\V(\rho_\tau | \trho_h^{k+1}) -\V(\rho_{i,h}^{k+1}| \trho_h^{k+1})) \leqslant \int_{\Omega} \nabla V[\trho_h^{k+1}] \cdot \xi \rho_h^{k+1} \, dx,
\end{eqnarray}

If we combine \eqref{first variation}, \eqref{first variation dist2}, \eqref{first variation entropie} and \eqref{first variation potentiel}, and if we replace $\xi$ by $-\xi$, we find that, for all $\xi \in \mathcal{C}_c^\infty(\Omega;\Rn)$,
\begin{equation}
\label{première variation borné}
\int_{\Omega} \nabla \varphi_{h}^{k+1} \cdot \xi \rho_{h}^{k+1} -h\int_{\Omega} P(\rho_{h}^{k+1}) \dive(\xi) +h\int_\Omega \nabla V[\trho_h^{k+1}] \cdot \xi \rho_{h}^{k+1}=0.
\end{equation}

Now we claim that $P(\rho_{h}^{k+1}) \in W^{1,1}(\Omega)$. Indeed, since $P$ is controlled by $F$ thanks to assumption \eqref{hyp F}, \eqref{estimation fonctional} gives $P(\rho_{h}^{k+1}) \in L^{1}(\Omega)$. Moreover, using \eqref{première variation borné}, we obtain

\begin{eqnarray*}
 \left| \int_\Omega P(\rho_{h}^{k+1}) \dive(\xi) \right| \leqslant \left[ \int_{\Omega} \frac{|\nabla \varphi_{h}^k(y)|}{h}\rho_{h}^{k+1} + \int_\Omega |\nabla V[\trho_h^{k+1}] | \rho_{h}^{k+1}  \right]  \|\xi \|_{L^{\infty}(\Omega)}. %\leqslant \left[ \frac{W_2(\rho_{i,h}^k,\rho_{i,h}^{k+1})}{h} +C\right] \|\xi\|_{L^{\infty}(\Omega)}.
\end{eqnarray*}
But using \eqref{hyp V cv L2}, \eqref{hyp V L2 bound}, \eqref{estimation distance} and Cauchy-Schwarz inequality we get
\begin{eqnarray*}
\int_\Omega |\nabla V[\trho_h^{k+1}] | \rho_{h}^{k+1} & \leqslant &  \left[ \int_\Omega |\nabla V[\trho_h^{k+1}] -\nabla V[\rho_h^{k+1}]| \rho_{h}^{k+1} +\int_\Omega |\nabla V[\rho_h^{k+1}] | \rho_{h}^{k+1} \right]\\
 & \leqslant &  \left[ \Big(\int_\Omega |\nabla V[\trho_h^{k+1}] -\nabla V[\rho_h^{k+1}]|^2\rho_{h}^{k+1}\Big)^{1/2} +\Big(\int_\Omega |\nabla V[\rho_h^{k+1}] |^2 \rho_{h}^{k+1} \Big)^{1/2} \right] \\
& \leqslant & C\left[ W_2(\trho_h^{k+1},\rho_h^{k+1}) +1 \right]\\
& \leqslant & C.
\end{eqnarray*}
We thus have
\begin{eqnarray*}
 \left| \int_\Omega P(\rho_{h}^{k+1}) \dive(\xi) \right| \leqslant \left[ \frac{W_2(\trho_{h}^{k+1},\rho_{h}^{k+1})}{h} +C\right] \|\xi\|_{L^{\infty}(\Omega)}.
\end{eqnarray*}

This implies $P(\rho_{h}^{k+1}) \in BV(\Omega)$ and $\nabla P(\rho_{h}^{k+1})=\left(-\nabla V[\trho_h^{k+1}] \rho_{h}^{k+1}-\frac{\nabla \varphi_{h}^{k+1}}{h} \rho_{h}^{k+1}\right)$ in $\M^n(\Omega)$. In fact, $P(\rho_{h}^{k+1})$ is in $W^{1,1}(\Omega)$ because $ \nabla V[\trho_h^{k+1}]\rho_{h}^{k+1} + \frac{\nabla \varphi_{h}^k}{h} \rho_{h}^{k+1} \in L^1(\Omega)$ and  we have proved \eqref{equality a.e 1}.

\end{proof}

We immediately deduce an $L^1((0,T), BV(\Omega))$ estimate for $P(\rho_h)$:
\begin{cor}
\label{estimation BV}
For all $T>0$, we have
\begin{eqnarray}
\label{estimation gradient}
\|P(\rho_{h})\|_{L^1((0,T);W^{1,1}(\Omega))}\leqslant CT.
\end{eqnarray} 
\end{cor}

\begin{proof}
If we integrate \eqref{equality a.e 1}, we obtain
$$ h\int_\Omega |\nabla P(\rho_{h}^{k+1})| \leqslant W_2(\rho_{h}^{k+1},\trho_{h}^{k+1}) +Ch, $$
Then we sum from $k=0$ to $N-1$ and thanks to \eqref{estimation distance}, we have
$$ \int_0^T \int_\Omega |\nabla P(\rho_{h})| \leqslant CT. $$
We conclude thanks to \eqref{hyp F} and \eqref{estimation fonctional}.
\end{proof}

\begin{prop}
\label{discreet equation}
Let $h>0$, $N\in \N^*$, $T:=Nh$,  $t_k:=hk$, for $k=0, \cdots, N$,  then, for every  $\phi \in \mathcal{C}^\infty_c ([0,T)\times \Rn)$
\begin{eqnarray*}
\int_0^T \int_{\Omega} \trho_{h}^2(t,x)  (\partial_t \phi(t,x) +W[\rho_h(t-h)] \cdot \nabla \phi(t,x) )\,dxdt&=&h\sum_{k=0}^{N-1}\int_{\Omega} \nabla P(\rho_{h}^{k+1}(x)) \cdot \nabla \phi (t_{k},x)\,dx\\
&+&h\sum_{k=0}^{N-1}\int_{\Omega} \nabla V[\trho_h^{k+1}] \cdot  \nabla \phi(t_{k},x)  \rho_h^{k+1} \, dx\\
&+&\sum_{k=0}^{N-1}\int_{\Omega \times \Omega} \mathcal{R}[\phi(t_{k},\cdot)](x,y) d\gamma_{h}^{k+1} (x,y)\\
&-& \int_{\Omega} \rho_{0}(x) \phi(0,x) \, dx,
\end{eqnarray*}
with, for all $\phi \in \mathcal{C}^\infty_c([0,T) \times \Rn)$,
$$ |\mathcal{R}[\phi](x,y)| \leqslant \frac{1}{2} \|D^2 \phi \|_{L^\infty ([0,T) \times \Omega)} |x- y|^2,$$
and $\gamma_{h}^{k+1}$ is the optimal transport plan in $\Pi (\rho_{h}^{k+1},\trho_{h}^{k+1})$.
\end{prop}

\begin{proof}
Let $\varphi \in  \mathcal{C}^\infty_c ( \Rn)$,  multiplying \eqref{equality a.e 1} by $\nabla \varphi$ and integrating on $\Omega$, we obtain 
$$ -\int_\Omega  \nabla \varphi_{h}^{k+1} \cdot \nabla \varphi \rho_{h}^{k+1}  = h \left( \int_\Omega \nabla P(\rho_h^{k+1}) \cdot \nabla \varphi + \int_\Omega \nabla V[\trho_h^{k+1}] \cdot \nabla \varphi \rho_{h}^{k+1} \right).$$

But, we can rewrite the left hand side by
$$ -\int_\Omega  \nabla \varphi_{h}^{k+1} \cdot \nabla \varphi \rho_{h}^{k+1} = \int_{\Omega\times \Omega} (y-x) \cdot \nabla \varphi(x) \,d\gamma_{h}^{k+1}(x,y).$$
A second-order Taylor-Lagrange formula then gives
\begin{eqnarray}
\int_{\Omega \times \Omega} (y-x) \cdot \nabla \varphi(x) \,d\gamma_{h}^{k+1} & = & \int_{\Omega \times \Omega } \left( \varphi(y) - \varphi(x) \right) \,d\gamma_{h}^{k+1}(x,y) - \int_{\Omega \times \Omega } \mathcal{R}[\varphi](x,y) d\gamma_{h}^{k+1} (x,y) \nonumber \\
& =& \int_\Omega \varphi ( \trho_h^{k+1} -\rho_h^{k+1} )- \int_{\Omega \times \Omega } \mathcal{R}[\varphi](x,y) d\gamma_{h}^{k+1} (x,y). \label{Taylor-Lagrange}
\end{eqnarray}

Now let $\phi \in \mathcal{C}^\infty_c([0,T) \times \Rn)$, we have

\begin{eqnarray*}
&&\int_0^T \int_{\Omega} \trho_{h}^2(t,x)  (\partial_t \phi(t,x) +W[\rho_h(t-h)] \cdot \nabla \phi(t,x) )\,dxdt \\ &=&\sum_{k=0}^{N-1} \int_{t_k}^{t_{k+1}} \int_\Omega \rho_h^k(x) (\partial_t \phi +W[\rho_h^k] \cdot \nabla \phi )(t,X_h^k(t-t_k,x)) \,dxdt\\
\end{eqnarray*}
But, on $[t_k,t_{k+1}]$,
$$ \frac{d}{dt}[\phi(t,X_h^k(t-t_k,x))]=(\partial_t \phi +W[\rho_h^k] \cdot \nabla \phi )(t,X_h^k(t-t_k,x)).$$
Then,
\begin{eqnarray*}
\sum_{k=0}^{N-1} \int_{t_k}^{t_{k+1}} \int_\Omega \rho_h^k(x) (\partial_t \phi +W[\rho_h^k] \cdot \nabla \phi )(t,X_h^k(t-t_k,x)) \,dxdt& = & \sum_{k=0}^{N-1} \int_{t_k}^{t_{k+1}} \int_\Omega \rho_h^k(x) \frac{d}{dt}[\phi(t,X_h^k(t-t_k,x))] \,dxdt\\
&=& \sum_{k=0}^{N-1}  \int_\Omega \rho_h^k(x) \left[ \phi(t_{k+1},X_h^k(h,x)) - \phi(t_k,x)\right] \,dx\\
&=& \sum_{k=0}^{N-1}  \int_\Omega  \left[ \phi(t_{k+1}, x) \trho_h^{k+1}(x) - \phi(t_k,x)\rho_h^k(x) \right] \, dx\\
&=& \sum_{k=0}^{N-1}  \int_\Omega  \phi(t_{k+1}, x) \left( \trho_h^{k+1}(x) - \rho_h^{k+1}(x) \right)\\
&& - \int_\Omega \phi(0,x) \rho_0(x) \, dx.
\end{eqnarray*}

Then the proof is complete by applying \eqref{Taylor-Lagrange} with $\varphi = \phi(t_{k+1}, \cdot)$.

\end{proof}

\section{Convergence and proof of Theorem \ref{eqution P2}}\label{sec-proof}

\subsection{Weak and strong convergences}
Using a refined version of Ascoli theorem (see \cite{AGS}), estimate \eqref{estimation distance} and remark \ref{estimation trho} and taking subsequences, if necessary, we have that, for every $T<+\infty$, $\rho_h$, $\trho_h^1$ and $\trho_h^2$ converge in $L^{\infty}((0,T), W_2)$ to some respective limits   $\rho$, $\trho^1$ and $\trho^2$:
\[\sup_{t\in[0,T]} \max (W_2(\rho_h(t,.), \rho(t,.)),   W_2(\trho_h^1(t,.), \trho^1(t,.)), W_2(\trho_h^2(t,.), \trho^2(t,.)))\to 0 \mbox{ as } h\to 0.\] 
 In fact, these three sequences have to converge to the same limit, $\rho$. Indeed, for all $\varphi \in \mathcal{C}^\infty_c((0,T) \times \Omega)$,
\begin{eqnarray*}
\int_0^T \int_\Omega \varphi (\rho_h -\trho_h^1) & = & \sum_{k=0}^{N-1} \int_{t_k}^{t_{k+1}} \int_\Omega \varphi (\rho_h^{k+1} - \trho_h^{k+1})\\
& \leqslant & Ch \sum_{k=0}^{N-1}W_2(\rho_h^{k+1} ,\trho_h^{k+1})\\
& \leqslant & ChN^{1/2}\left( \sum_{k=0}^{N-1}W_2^2(\rho_h^{k+1} ,\trho_h^{k+1}) \right)^{1/2}\\
&\leqslant & CT^{1/2}h,
\end{eqnarray*}
because of \eqref{estimation distance}. With a similar computation, we  find that $\rho_h$ and $\trho_h^2$ converge to the same limit. We thus have
\begin{equation}\label{cvuwass}
\sup_{t\in[0,T]} \max (W_2(\rho_h(t,.), \rho(t,.)),   W_2(\trho_h^1(t,.), \rho(t,.)), W_2(\trho_h^2(t,.), \rho(t,.)))\to 0 \mbox{ as } h\to 0.
\end{equation}

Moreover it is classical to deduce from  \eqref{estimation distance} and remark \ref{estimation trho} an H\"older-like  estimate of the form \\$W_2(\rho_h(t,.), \rho_h(s,.)) \le C \sqrt{ \vert t-s\vert +h}$ from which one deduces that the limit curve $\rho$ actually belongs to $ \mathcal{C}^{1/2}((0,T), W_2)$. This kind of convergence will be enough to pass to the limit in $\nabla V[\trho_h^1] \rho_h$ and $W[\rho_h]\trho_h^2$, because of assumptions \eqref{hyp V cv L2} and \eqref{hyp W cv L2}, but we will need a stronger convergence to deal with the nonlinear diffusion term $P(\rho_h)$. Fo this purpose, we will use an extension of the Aubin-Lions Lemma due to Rossi and Savar\'e in \cite{RS}:

\begin{thm}[th. 2 in \cite{RS}]
\label{Rossi-Savaré}
On a Banach space $B$, let be given
\begin{itemize}
\item a normal coercive integrand $\G \, : \, B \rightarrow \R^+$, i.e, $\G$ is l.s.c and its sublevels are relatively compact in $B$,
\item a pseudo-distance $g \, : \, B\times B \rightarrow [0,+\infty]$, i.e, $g$ is l.s.c and $ \left[ g(\rho,\mu)=0, \, \rho,\mu \in B \text{ with } \G(\rho),\G(\mu) < \infty \right] \Rightarrow \rho=\mu$.
\end{itemize}

Let $T>0$ and $U$ be a set of measurable functions $u \, : \, (0,T) \rightarrow B$. Under the hypotheses that 
\begin{align}
\label{condition U}
\sup_{u\in U} \int_0^T \G(u(t)) \,dt <+\infty \qquad \text{  and  } \qquad \lim_{h\searrow 0} \sup_{u \in U} \int_0^{T-h} g(u(t+h),u(t)) \, dt =0,
\end{align} 
$U$ contains a subsequence $(u_n)_{n \in \N}$ which converges (strongly in $B$) in measure with respect to $t\in (0,T)$ to a limit $u_\star \, : \, (0,T) \rightarrow B$.
\end{thm}

We now apply this theorem to $B=L^1(\Omega)$, $U=\{ \rho_h\}_h$, $g$ defined by
$$ g(\rho,\mu) :=\left\{ \begin{array}{ll}
W_2(\rho,\mu) & \text{if } \rho,\mu \in \mathcal{P}_2(\Omega),\\
+\infty & \text{otherwise},
\end{array}\right.$$
and $\G$ by
$$ \G(\rho):= \left\{ \begin{array}{ll}
\F(\rho)+ \|P(\rho) \|_{BV(\Omega)} + M(\rho)  & \text{if } \rho \in \mathcal{P}_{2}^{ac}(\Omega), P(\rho) \in BV(\Omega) \text{ and } F(\rho) \in L^1(\Omega),\\
+\infty & \text{otherwise}.
\end{array}\right.$$

\begin{lem}
\label{sublevels of G}
$\G$ is l.s.c on $L^1(\Omega)$ and its sublevels are relatively compact in $L^1(\Omega)$.
\end{lem}

\begin{proof}

Let us start by proving that sublevels of $\G$ are relatively compact in $L^1(\Omega)$. Let 
\[A_c:= \left\{ \rho \in L^1(\Omega) \, : \, \G(\rho) \leqslant c \right\}\]
and  $(\rho_k)$ be a sequence in $A_c$ then $P(\rho_k)$ is bounded in $BV(\Omega)$ thus, up to a subsequence $P(\rho_k)$ converges to some $\Phi$ in $L^1_{\loc}(\Omega)$ and a.e.. Since $P$ is continuous, one to one and its inverse is continuous, $\rho_k$ converges to $\rho:=P^{-1}(\Phi)$  a.e.; and, since $\G(\rho_k) \leqslant c$ and $F$ is superlinear, $\rho_k$ is uniformly integrable, using Vitali's convergence theorem, we obtain that $\rho_k$ converges to $\rho$ in $L^1(K\cap \Omega)$ for every compact $K$. To conclude that there is convergence in $L^1(\Omega)$, we use the fact that the second momentum of $\rho_k$ is uniformly bounded:
\begin{eqnarray*}
\int_\Omega |\rho_k - \rho | & \leqslant & \int_{\Omega \setminus B_R} \frac{|x|^2}{R^2}|\rho_k - \rho | + \int_{B_R \cap \Omega} |\rho_k - \rho | \\
& \leqslant & \frac{2c}{R^2} + \int_{B_R\cap \Omega} |\rho_k - \rho |. 
\end{eqnarray*}
The first term in the right hand can be made arbitrary small by choosing $R$ large enough and the second term converges to zero by $L^1(B_R \cap \Omega)$-convergence.

Now we have to show the lower semi-continuity of $\G$ on $L^1(\Omega)$. Let $(\rho_k)$ be a sequence which converges strongly to $\rho$ in $L^1(\Omega)$ with (without loss of generality)  $\sup_{k} \G(\rho_k) \leqslant C<+\infty$. Without loss of generality, we can assume that $\rho_k$ converges to $\rho$ a.e. Since $\sup_{k} \G(\rho_k) \leqslant C$, $P(\rho_k)$ is uniformly bounded in $BV(\Omega)$ so $P(\rho_k)$ converges weakly to $\mu$ in $BV(\Omega)$. Moreover, $P(\rho_k)$ converges strongly to $\mu$ in $L^1_{\loc}(\Omega)$. We can conclude that $\mu=P(\rho)$ and by lower semi-continuity of $\F$, $M$ and the $BV$-norm we have
$$ \G(\rho) \leqslant \liminf_{k\nearrow +\infty} \G(\rho_k).$$

\end{proof}

Thanks to lemma \ref{sublevels of G}, to apply theorem \ref{Rossi-Savaré}, it remains to verify \eqref{condition U}. The first condition of \eqref{condition U} is satisfied because of the estimate on the momentum, \eqref{estimation moment}, on the internal energy $\F$, \eqref{estimation fonctional} and on the gradient of $P(\rho_h)$ \eqref{estimation gradient}. The second condition of  \eqref{condition U} comes from the estimate on the distance \eqref{estimation distance} and remark \ref{estimation trho} (see for example \cite{DFM} for a detailed proof).
Then theorem \ref{Rossi-Savaré} implies that $\rho_h$ converges in measure with respect to $t$ in $L^1(\Omega)$ to $\rho$. Since convergence in measure implies a.e convergence, up to a subsequence, we may also assume
that $\rho_{h}(t,.)$ converges strongly in $L^1(\Omega)$ to $\rho(t,.)$ for a.e. $t$. Then Lebesgue's dominated convergence theorem implies that $\rho_{h}$ converges strongly in $L^1((0,T) \times \Omega)$ to $\rho$. 

\smallskip

Thanks to  \eqref{hyp F} and \eqref{estimation fonctional} $P(\rho_h)$ is  uniformly bounded in $L^\infty((0,T),L^1(\Omega))$. In addition, by corollary \ref{estimation BV}, $P(\rho_h)$ is uniformly bounded in $L^1((0,T),W^{1,1}(\Omega))$.  Thanks to the  Sobolev embedding, we deduce that $P(\rho_h)$ is uniformly bounded in $L^\infty((0,T),L^1(\Omega)) \cap L^1((0,T),L^{n/n-1}(\Omega))$. To have uniform integrability of $P(\rho_h)$ both in the time and space variables, the following will be  useful:

%Now we want to prove that $P(\rho_h)$ is uniformly bounded in $L^\infty((0,T),L^1(\Omega))$, using \eqref{hyp F}, and thanks to BV-argument, 

\begin{lem}
\label{interpolation}
Let $p>1$, $q:=\frac{2p-1}{p}$ and $f \in L^\infty((0,T),L^1(\Omega)) \cap L^1((0,T),L^{p}(\Omega))$ then $f \in L^q((0,T) \times \Omega)$  and we have
$$ \| f \|_{L^q((0,T) \times \Omega)}^q \leqslant  \|f \|^{q-1}_{L^\infty_t(L^1_x)} \|f\|_{L^1_t(L^p_x)}.$$
\end{lem}

\begin{proof}
Writing
\[\frac{1}{q}=\frac{p}{2p-1}= \theta + \frac{1-\theta}{p}, \; \mbox{ with } \theta=\frac{p-1}{2p-1}\in (0,1)\]
and observing that $(1-\theta)q=1$ and $\theta q=q-1$, the interpolation inequality yields that 
$$ \|f \|_{L^q_x}^q  \leqslant \|f\|_{L^1_x}^{ q-1} \|f\|_{L^p_x},$$
since $f \in L^\infty((0,T),L^1(\Omega)) \cap L^1((0,T),L^{p}(\Omega))$ this implies that $f \in L^q((0,T) \times \Omega)$ and
\begin{eqnarray*}
 \| f \|_{L^q((0,T) \times \Omega)}^q= \int_0^T \|f \|_{L^q_x}^q=  & \leqslant &  \int_0^T\|f\|_{L^1_x}^{q-1} \|f\|_{L^p_x}\\
&\leqslant & \|f\|_{L^\infty_t(L^1_x)}^{q-1} \|f\|_{L^1_t(L^p_x)}.
\end{eqnarray*}
\end{proof}

%\red{Peut être mettre le lemme dans les préliminaires}\\

Applying lemma \ref{interpolation} we deduce that $P(\rho_h)$ is uniformly bounded in $L^{(n+1)/n}((0,T)\times \Omega)$. This implies that $P(\rho_h)$ is uniformly integrable and since we know that it converges a.e. to $P(\rho)$, we can deduce from  Vitali's convergence theorem  that $P(\rho_h)$ converges strongly to $P(\rho)$ in $L^1_{\loc}( (0,T) \times \Omega)$. 

Thanks to Corollary \ref{estimation BV} we deduce that  $\nabla P(\rho_h)$ converges vaguely to $\nabla P(\rho)$ in $\mathcal{M}^n_{\loc}((0,T) \times \Omega)$. In fact, we have $\nabla P(\rho)$ in $\mathcal{M}^n((0,T) \times \Omega)$ and narrow convergence of $\nabla P(\rho_h)$ to $\nabla P(\rho)$, thanks to Prokhorov Theorem and the following tightness estimate:

\begin{lem}\label{tightgradp}
The family  $\nabla P(\rho_h)$, viewed as  vector-valued measures  on $[0,T]\times \Omega$, is tight, more precisely, for every $h$ and every $A$ measurable, $A\subset \Omega$
\begin{equation}\label{tightquant}
\int_0^T \int_A \vert \nabla P(\rho_h) \vert \le C\Big(1+\sqrt{h}\Big) \Big(\int_0^T \int_A \rho_h(t,x) \mbox{d} x \mbox{d}t\Big)^{1/2}.
\end{equation}
\end{lem}

\begin{proof}
Integrating  \eqref{equality a.e 1} on $(0,T)\times A$ together with Cauchy Schwarz inequality and \eqref{estimation distance}, we get (taking $N=\lceil \frac{T}{h } \rceil+1$, say) 
\[\begin{split}
\int_0^T \int_A \vert \nabla P(\rho_h) \vert  &\leqslant \sum_{k=0}^N \int_A \vert \nabla \varphi_h^{k+1} \vert \rho_h^{k+1} + \int_0^T \int_A \vert \nabla V[\trho_h^1] \vert \rho_h \\
&\leqslant \sum_{k=0}^N \Big(\int_\Omega \vert \nabla \varphi_h^{k+1} \vert^2 \rho_h^{k+1}\Big)^{1/2} \Big(\int_A \rho_h^{k+1}\Big)^{1/2} + \int_0^T \int_A \vert \nabla V[\trho_h^1] \vert \rho_h\\
& \leqslant  \Big(\sum_{k=0}^N W_2^2(\trho_h^{k+1}, \rho_h^{k+1})\Big)^{1/2} \Big(\int_0^T \int_A \rho_h(t,x) \mbox{d} x \mbox{d}t\Big)^{1/2}+ \int_0^T \int_A \vert \nabla V[\trho_h^1] \vert \rho_h\\
& \leqslant  C \sqrt{h} \Big(\int_0^T \int_A \rho_h(t,x) \mbox{d} x \mbox{d}t\Big)^{1/2}+ \int_0^T \int_A \vert \nabla V[\trho_h^1] \vert \rho_h.
\end{split}\]
Moreover, with Cauchy Schwarz inequality, \eqref{hyp V L2 bound} and \eqref{hyp V cv L2}, we also have
\[\begin{split}
\int_0^T \int_A \vert \nabla V[\trho_h^1] \vert \rho_h &\leqslant  \int_0^T \int_A \vert \nabla V[\rho_h]\vert \rho_h + \int_0^T \int_A \vert \nabla V[\trho_h^1] -\nabla V[\rho_h]\vert \rho_h \\
 & \leqslant C\Big(1+\sqrt{h}\Big) \Big(\int_0^T \int_A \rho_h(t,x) \mbox{d} x \mbox{d}t\Big)^{1/2}
\end{split}\]
which proves \eqref{tightquant}.  The tightness of $\nabla P(\rho_h)$ therefore immediately follows from that of $\rho_h$ and \eqref{tightquant}.

\end{proof}

We can summarize all of this in the next result:

\begin{thm}\label{cvgcesummary}
Up to a subsequence  $\rho_h$ converges strongly in $L^1((0,T) \times \Omega)$, $P(\rho_h)$ converges strongly to $P(\rho)$ in $L^1_{\loc}( (0,T) \times \Omega)$ and $\nabla P(\rho_h)$ converges to $\nabla P(\rho)$ narrowly in $\mathcal{M}((0,T) \times \Omega)$.
\end{thm}

\subsection{End of the proof of theorem \ref{eqution P2}}

In this section we finish the proof of  theorem \ref{eqution P2}. We have to pass to the limit in all terms in proposition \ref{discreet equation}. The linear term (with time derivative) and the diffusion term converge to the desired result because $\trho_{h}^2$ converges to $\rho$  in $L^{\infty} ([0,T] , W_2)$ and $\nabla P(\rho_{h})$ converges to $\nabla P(\rho)$ narrowly in $\mathcal{M}^n([0,T] \times \Omega)$. Remainder term goes to $0$ when $h$ goes to $0$ because of \eqref{estimation distance}. 

So we just have to check the convergence in the  transport terms, in what follows the test-function $\phi$ belongs again to  $\mathcal{C}^\infty_c([0,T) \times \Rn)$. 

\begin{itemize}
\item{\it term in $W$: \it} We have to show:
$$\int_0^T \int_{\Omega} \trho_{h}^2(t,x) W[\rho_h(t-h)](x) \cdot \nabla \phi(t,x) \,dxdt \rightarrow \int_0^T \int_{\Omega} \rho(t,x) W[\rho(t,\cdot)](x) \cdot \phi(t,x) \,dxdt.$$

We first have
\begin{eqnarray*}
\int_0^T \int_{\Omega} \trho_{h}^2(t,x) W[\rho_h(t-h)](x) \cdot \nabla \phi(t,x) \,dxdt & =  & \int_0^T \int_{\Omega} \trho_{h}^2(t,x) W[\trho_h^2(t)](x) \cdot \nabla \phi(t,x) \,dxdt\\
& + &\int_0^T \int_{\Omega} \trho_{h}^2(t,x) ( W[\rho_h(t-h)](x)- W[\trho_h^2(t)](x) )\cdot \nabla \phi(t,x) \,dxdt.
\end{eqnarray*}

The second term in the right hand side goes to zero when $h$ goes to $0$. Indeed,
\begin{multline*}
\Big \vert \int_0^T \int_{\Omega} \trho_{h}^2(t,x) ( W[\rho_h(t-h)](x)- W[\trho_h^2(t)](x) )\cdot \nabla \phi(t,x) \,dxdt \Big \vert \\ 
 \leqslant  C\int_0^T \left( \int_{\Omega} \trho_{h}^2(t,x) | W[\rho_h(t-h)](x)- W[\trho_h^2(t)(x)] |^2 \,dx \right)^{1/2} dt,
\end{multline*}
then using \eqref{hyp W cv L2},
\begin{eqnarray*}
\Big \vert \int_0^T \int_{\Omega} \trho_{h}^2(t,x) ( W[\rho_h(t-h)](x)- W[\trho_h^2(t)](x) )\cdot \nabla \phi(t,x) \,dxdt \Big \vert & \leqslant & C\int_0^T W_2(\rho_h(t-h),\trho_h^2(t)) \, dt \\
&\leqslant & C\sum_{k=0}^{N-1} \int_{t_k}^{t_{k+1}}W_2(\rho_h^k,X_h^k(t-t_k)_{\#}\rho_h^k) \,dt\\
& \leqslant & CTh,
\end{eqnarray*}
because of \eqref{estimation distance transport libre}. Moreover,  using   \eqref{hyp V cv L2}, we get
\begin{multline*}
\Big \vert \int_0^T \int_{\Omega} (\trho_{h}^2(t,x) W[\trho_h^2(t)](x) \cdot \nabla \phi(t,x) - \rho(t,x) W[\rho(t)](x) \cdot \nabla \phi(t,x)) \,dxdt\Big  \vert\\
  \leqslant  C  \int_0^T \int_{ \Omega} \trho_{h}^2(t,x)\vert W[\trho_h^2(t)](x)-W[\rho(t)](x)) \vert \,dxdt\\
  +   \Big \vert \int_0^T \int_{\Omega} ( \rho(t,x)- \trho_{h}^2(t,x))  W[\rho(t)](x) \cdot \nabla \phi(t,x)  \,dxdt \Big \vert\\
 \leqslant  CT\sup_{t\in [0,T]} W_2(\trho_h^2(t),\rho(t))+ \Big \vert \int_0^T \int_{\Omega} ( \rho(t,x)- \trho_{h}^2(t,x))  W[\rho(t)](x) \cdot \nabla \phi(t,x)  \,dxdt \Big \vert
\end{multline*}
the first term in the right hand-side converges to $0$ because of \eqref{cvuwass}. As for the second one, it also converges to $0$, because  $W[\rho] \cdot \nabla \phi$ belongs to $L^{\infty}((0,T)\times \Omega)$ and $\trho_h^2$ is uniformly integrable by remark \ref{estimation trho} and the superlinearity of $F$, hence, up to a subsequence it converges to $\rho$ weakly in $L^1((0,T)\times \Omega)$.

\item{\it term in $\nabla V$: \it} We claim that
$$ h\sum_{k=0}^{N-1}\int_{\Omega} \nabla V[\trho_h^{k+1}] (x)\cdot  \nabla \phi(t_{k},x)  \rho_h^{k+1} \, dx \rightarrow \int_0^T \int_\Omega \nabla V[\rho(t,\cdot)](x) \cdot  \nabla \phi(t,x)  \rho(t,x) \, dxdt.$$
The proof is the same as the  previous one   for $W$, using \eqref{hyp V cv L2}, \eqref{hyp V Linffty bound} and the convergence of $\rho_h$ to $\rho$ in $L^1((0,T)\times \Omega))\cap L^{\infty}((0,T), W_2)$.

\end{itemize}

\section{On extension to systems and uniqueness}\label{sec-concl}

The splitting transport-JKO scheme described above, can easily be adapted, under suitable assumptions to the case of systems for the evolution of $N$ species coupled by nonlocal drifts:
\begin{equation}\label{systgal}
\partial_t \rho_i -\Delta P_i(\rho_i)-\dive(\rho_i U_i[\rho_1, \cdots, \rho_N])=0, \; \rho_i(0,.)=\rho_{i,0}, \; i=1, \cdots, N,
\end{equation}
where $P_i(s)=s F'_i(s)-F_i(s)$ is the pressure associated to a strictly convex superlinear function $F_i$ with corresponding internal energy $\F_i(\rho_i):=\int F_i(\rho_i(x)) dx$. 
Decomposing each drift  $U_i[\rho_1, \cdots, \rho_N]=\nabla V_i [\rho_1, \cdots, \rho_N] -W_i[\rho_1, \cdots, \rho_N]$ with $\dive(W_i[\rho_1, \cdots, \rho_N])=0$ and under similar assumptions as in paragraph \ref{sec-assump}, one can show, by similar arguments as above, convergence as $h\to 0$ to a solution of \eqref{systgal} of the following  splitting scheme. 

\medskip

Starting form $\rho_{i,h}^0=\rho_{i,0}$ and given $\rho_h^k=(\rho_{1,h}^k, \cdots \rho_{N,h}^k)$ we find $\rho_h^{k+1}=(\rho_{1,h}^{k+1}, \cdots \rho_{N,h}^{k+1})$ by:
\begin{itemize}

\item setting $\trho_{i,h}^{k+1}=X_{i,h}^k(h,.)_\# \rho_{i,h}^k$ where 
\[\partial_t X_{i,h}^k = W_i[\rho_h^k] \circ X_{i,h}^k,\; X_{i,h}^k(0,\cdot) = \id,\]

\item defining  $\trho_{h}^{k+1}=(\trho_{1,h}^{k+1}, \cdots,  \trho_{N,h}^{k+1})$, $\rho_h^{k+1}=(\rho_{1,h}^{k+1}, \cdots \rho_{N,h}^{k+1})$ is obtained by the semi-implicit JKO scheme:
\[ \rho_{i,h}^{k+1} =\argmin_{\rho_i \in \Paad(\Omega)}   \left\{ W_2^2(\rho_i,\tilde{\rho}_{i,h}^{k+1}) +2h\left( \F_i(\rho_i) +\int_{\Omega}  V_i [\tilde{\rho}_h^{k+1}] \rho_i \right)\right\}.\]

\end{itemize}

%\red{je ne sais pas si ca vaut la peine de donner plus de detail}

Finally, let us say a few words on uniqueness which we have not addressed here, but which can be obtained at least in two ways: either by assuming some displacement semiconvexity of the internal energy and proving some exponential in time contraction estimate on  the $W_2$ distance  between two solutions (see \cite{DFF,L}), or by assuming some nondegeneracy of the diffusion and establishing some $H^{-1}$ contraction estimate (see \cite{CL}).

%\subsection{On uniqueness}

%\red{to do}

%\section{General cost}
%In this section, we want to solve 
%\begin{eqnarray}
%\label{eqution P}
%\partial_t \rho - \dive(\rho \nabla c^*(\nabla F'(\rho)))- \dive(\rho U [\rho] ) = 0, \qquad \rho_{|t=0} = \rho_0,
%\end{eqnarray}
%on $(0,+\infty) \times \Omega$, where $\Omega$ is a convex bounded subset of $\Rn$, equipped with natural Neumann boundary condition.


\begin{thebibliography}{999}
%\addcontentsline{toc}{section}{Bibliographie}
\bibitem {A} M. Agueh, {\it Existence of solutions to degenerate parabolic equations via Monge-Kantorovich theory}, Adv. Differential Equations \textbf{10}(3),  309-360 (2005). 
\bibitem {AGS} L. Ambrosio, N. Gigli and G. Savar\'e, {\it Gradient flows in metric spaces and in the space of probability measures}, Lectures in Math., ETH Z\" urich, 2005.
\bibitem {B} Y. Brenier, {\it Polar factorization and monotone rearrangement of vector-valued functions}, Comm. Pure Appl. Math. \textbf{44} (4), 375--417 (1991).
\bibitem {BB} J.-D. Benamou and Y. Brenier, {\it A computational fluid mechanics solution to the Monge-Kantorovich mass transfer problem}, Numer. Math. \textbf{84}(3), 375--393 (2000).
%\bibitem {BC} J-B Baillon, G. Carlier, {\it From discrete to continuous Wardrop equilibria, Networks and Heterogenous Media}, 7(2), 2012.
\bibitem {CL} G. Carlier and M. Laborde, {\it Remarks on continuity equations with nonlinear diffusion and nonlocal drifts}, preprint 2016.
\bibitem{CDFLS} J. A. Carrillo, M. DiFrancesco, A. Figalli, T. Laurent, and D. Slepcev,{\it  Global-in-time weak measure solutions and finite-time aggregation for nonlocal interaction equations}, Duke Math. J. \textbf{156}(2), 229--271 (2011).

%\bibitem {CLM} G. Crippa and M. L\'ecureux-Mercier, {\it existence and uniqueness of measure solutions for a system of continuity equations with non-local flow}, 2011.
%\bibitem {DS} S. Daneri and G. Savar\'e, {\it Eulerian calculus for the displacement convexity in the Wasserstein distance}. SIAM J. Math. Anal., 40(3):1104-1122, 2008.
%\bibitem {DF} D.G. de Figueiredo, {\it Lectures on the Ekeland variational principle with applications and detours}, Tata Institute of Fundamental Research Lectures on Mathematics and Physics, 81. Published for the Tata Institute of Fundamental Research, Bombay; by Springer-Verlag, Berlin, (1989).
\bibitem {JKO} R. Jordan, D. Kinderlehrer and F. Otto, {\it The Variational Formulation of the Fokker-Plank Equation}, SIAM J. of Math. Anal. \textbf{29}, 1-17 (1998).
\bibitem {DFF} M. Di Francesco and S. Fagioli, {\it Measure solutions for nonlocal interaction PDEs with two species}, Nonlinearity \textbf{26}, 2777-2808 (2013).
\bibitem {DFM} M. Di Francesco and D. Matthes, {\it curves of steepest descent are entropy solutions for a class of degenerate convection-diffusion equations},  Calc. Var. and PDEs \textbf{50}(1-2), 199-230 (2012).
\bibitem {DPL} R.J DiPerna and P.L Lions, {\it Ordinary differential equations, transport theory and Sobolev spaces}, Invent. Math. \textbf{98}, 511-547, (1989). 
%%\bibitem {G} E. Giusti, {\it Minimal surfaces and functions of bounded variation}, vol. 80 of Monographs in Mathematics, Birkhauser Verlag, Basel, 1984.
\bibitem{L} M. Laborde, {\it On some non linear evolution systems which are perturbations of
Wasserstein gradient flows}, preprint 2015.
\bibitem {MC} R.-J. McCann, {\it A convexity principle for interacting gases}, Adv. Math. \textbf{128}, 153-179 (1997).

\bibitem{MesSan} A.R. M\'esz\'aros and F. Santambrogio, {\it Advection-diffusion equations with density constraints}, to appear in Analysis and PDEs (2016). 

%\bibitem {MMCS} D. Matthes, R. J. McCann and G. Savar\'e, {\it A family of nonlinear fourth order equations of gradient flow type. Comm. Partial Differential Equations}, 34(10-12):1352-1397, 2009.
\bibitem{O} F. Otto, {\it The geometry of dissipative evolution equations: the porous medium equation}, Comm. Partial Differential Equations \textbf{26}, 101-174
(2001).

%\bibitem {O} F. Otto, {\it Doubly degenerate diffusion equations as steepest descent}, Manuscript (1996).
\bibitem {RS} R. Rossi and G. Savar\'e, {\it tightness, integral equicontinuity and compactness for evolution problem in Banach spaces}, Ann. Sc. Norm. Super. Pisa Cl. Sci. (5), 2(2), 395-431 (2003).
\bibitem {S} F. Santambrogio, {\it Optimal Transport for Applied Mathematicians}, Birk\"auser Verlag, Basel, 2015.
\bibitem {V1} C. Villani, {\it Topics in optimal transportation}, volume 58 of graduate Studies in Mathematics. American Mathematical Society, Providence, RI, 2003.
\bibitem {V2} C. Villani, {\it Optimal transport: Old and New}, Springer Verlag (Grundlehren der mathematischen Wissenschaften), 2008.
\end{thebibliography}
\end{document}